\documentclass[11pt,reqno]{amsart}
\usepackage[all]{xy}
\usepackage[utf8]{inputenc}
\usepackage{hyperref}
\usepackage{amssymb}
\usepackage{graphicx}
\numberwithin{equation}{section}

\newcommand{\C}{{\mathbb C} }
\newcommand{\R}{{\mathbb R} }
\newcommand{\cA}{{\mathcal A} }
\newcommand{\cB}{{\mathcal B} }
\newcommand{\cC}{{\mathcal C} }
\newcommand{\cD}{{\mathcal D} }
\newcommand{\cE}{{\mathcal E} }
\newcommand{\cF}{{\mathcal F} }
\newcommand{\cG}{{\mathcal G} }

\newcommand{\cL}{{\mathcal L} }

\newcommand{\cN}{{\mathcal N} }
\newcommand{\cO}{{\mathcal O} }

\newcommand{\cX}{{\mathcal X} }

\newcommand{\cH}{{\mathcal H} }
\newcommand{\cK}{{\mathcal K} }
\newcommand{\cR}{{\mathcal R} }

\newcommand{\wt}{\widetilde}
\newcommand{\pt}{\partial}

\def\inv#1{{#1}^{-1}}
\def\ol#1{{\overline{#1}}}
\def\ul#1{{\underline{#1}}}
\newtheorem{theorem}{Theorem}[section]
\newtheorem{definition}{Definition}[section]
\newtheorem{lemma}{Lemma}[section]

\newtheorem*{remark*}{Remark}
\newtheorem{proposition}{Proposition}[section]
\newtheorem{corollary}{Corollary}[section]

\def\ks{Ko\-dai\-ra-Spen\-cer }
\def\ka{K{\"a}h\-ler }
\def\he{Her\-mite-Ein\-stein }
\def\kh{Ko\-ba\-ya\-shi-Hit\-chin }
\def\wp{Weil-Pe\-ters\-son }
\def\ii{\sqrt{-1}}
\def\idb{\sqrt{-1}\partial\overline{\partial} }
\def\C{\mathbb{C}}

\def\cinf{C^\infty}
\def\rk{{\mathrm{rk}}}
\def\tr{{\mathrm{\, tr\, }}}

\def\db{{\ol\partial}}

\def\we{\wedge}

\def\Hom{\mathrm{Hom}}
\def\End{\mathrm{End}}
\def\mspace{\mbox{\textvisiblespace}}

\begin{document}

\title[Differential geometry of moduli spaces of quiver bundles]{Differential
geometry of moduli spaces of quiver bundles}

\author[I.\ Biswas]{Indranil Biswas}

\address{School of Mathematics, Tata Institute of Fundamental Research, Homi
Bhabha Road, Mumbai 400005, India}

\email{indranil@math.tifr.res.in}

\author[G.\ Schumacher]{Georg Schumacher}

\address{Fachbereich Mathematik und Informatik,
Philipps-Universit\"at Marburg, Lahnberge, Hans-Meerwein-Straße, D-35032
Marburg, Germany}

\email{schumac@mathematik.uni-marburg.de}



\begin{abstract}
Let $P$ be a parabolic subgroup of a semisimple affine algebraic group $G$ defined
over $\mathbb C$ and $X$ a compact K\"ahler manifold. L.
\'Alvarez-C\'onsul and O. Garc\'{i}a-Prada associated to these a quiver $Q$ and
representations of $Q$ into holomorphic vector bundles on $X$ (\cite{A-G03},
\cite{A-G03b}). Our aim here is to investigate the differential geometric
properties of the moduli spaces of representations of $Q$ into vector bundles on
$X$. In particular, we construct a Hermitian form on these moduli spaces. A
fiber integral formula is proved for this Hermitian form; this fiber integral formula
implies that the Hermitian form is K\"ahler. We compute the curvature of this
K\"ahler form. Under an assumption which says that $X$ is a complex projective
manifold, this K\"ahler form is realized as the curvature of a certain
determinant line bundle equipped with a Quillen metric.
\end{abstract}

\maketitle

\tableofcontents

\section{Introduction}

Fix a semisimple affine algebraic group $G$ defined over $\mathbb C$ and a proper
parabolic subgroup $P\, \subset\, G$.
L. \'Alvarez-C\'onsul and O. Garc\'{i}a-Prada constructed a quiver $Q$ with relations
and established an equivalence of categories between the following two:
\begin{enumerate}
\item holomorphic finite dimensional representations of $P$;

\item finite dimensional representations of $Q$
\end{enumerate}
(see \cite{A-G03}). Let $X$ a compact connected K\"ahler manifold
equipped with a K\"ahler form. In \cite{A-G03},
the authors proved the following more general result:
There is an equivalence of categories between the following two:
\begin{enumerate}
\item $G$--equivariant holomorphic vector bundles on $X\times (G/P)$;

\item representations of $Q$ into holomorphic vector bundles on $X$.
\end{enumerate}

The representations of $Q$ into holomorphic vector bundles on $X$ are called
quiver bundles on $X$. The notion of stability of vector bundles on $X$ extends
to quiver bundles on $X$. It is known that polystable quiver bundles carry a
generalization of the Hermite-Einstein metrics \cite{A-G03}, \cite{A-G03b} that are solutions of certain equations of {\em vortex type}.

Given a holomorphic family of stable quiver bundles on $X$ parameterized by a
complex analytic space $T$, we construct a Hermitian structure on $T$; see
Definition \ref{dwp1}. This construction entails building the deformations of
quiver bundles (this is carried out in Section \ref{sec3.2}) and relating the
curvature of a generalized Hermite-Einstein metric to deformations
(this is carried out in Section \ref{se3.3}).

One of the main results here is a fiber integral formula for the above mentioned
Hermitian structure on $T$; see Proposition \ref{propfb1} and Proposition
\ref{propfb2}. As a corollary we obtain that the Hermitian structure is
K\"ahler (Corollary \ref{corfb1}).

Another main result is a computation of this K\"ahler form on the moduli
space of stable quiver bundles; this is carried out in Theorem \ref{tf}.

Now assume that $X$ is a complex projective manifold and the K\"ahler form on it
is integral. We construct a holomorphic Hermitian line bundle on the moduli
space of stable quiver bundles with the property that the corresponding
Chern form coincides with the K\"ahler form on the moduli space; see Theorem
\ref{th2}. The construction of this holomorphic Hermitian line bundle is modeled
on the works Quillen, \cite{Qu-85}, and Bismut--Gillet--Soul\'e \cite{bgs}.

\section{Quiver bundles -- Notation and fundamental properties}

\subsection{Representations of quivers}
We follow the notation of \cite{A-G03}. Two sets $Q_0$ and $Q_1$ together with two maps
$h,t : Q_1 \,\longrightarrow\, Q_0$ give rise to a {\em directed graph} or {\em quiver}. The elements of $Q_0$ are called vertices, and the elements of $Q_1$ are called arrows. For $a\in Q_1$ the vertices $ha := h(a)$ and $ta := t(a)$ are called head of $a$ and tail of $a$ respectively. For an arrow $a$ there is the notation $a\,:\, v \,\to\, v'$, where $v=ta$ and $v'=ha$. At this point the set of vertices $Q_0$ is not assumed to be finite. However the quiver $Q:=(Q_0,Q_1)$ will have to be {\em locally finite} meaning\ for any vertex $v\in Q_0$ the sets
$h^{-1}(v)$ and $t^{-1}(v)$ are assumed to be finite. A (non-trivial) path in $Q$ is a sequence $p=a_0\cdots a_m$ of arrows $a_j \in Q_1$ such that $ta_{i-1}=ha_i$ for $i=1,\ldots, m$:
\begin{equation}\label{eq:path}
p:\; \stackrel{tp}{\bullet}\; \stackrel{a_m}{\longrightarrow}\; \bullet \; \stackrel{a_{m-1}}{\longrightarrow} \; \cdots \; \stackrel{a_0}{\longrightarrow}\; \stackrel{hp}{\bullet}.
\end{equation}
The vertices $tp= ta_m$ and $hp=ha_0$ are respectively called tail and head of the path $p$. By definition, the trivial path $e_v$ at $v\in Q_0$ is equal to $e_v : v \to v$ in the above alternative notation.

A formal, finite sum $$r\,=\, c_1 p_1 + \ldots + c_\ell p_\ell$$ of paths $p_j$ with complex coefficients is called a {\em (complex) relation} of a quiver. A {\em quiver with
relations} is a pair $(Q,\cK)$, where $Q$ is a quiver with $\cK$ being a set of relations.

A {\em linear representation} $\mathbf{ R =(V,\pmb \varphi)}$ of a quiver $Q$ is given by a collection $\mathbf V$ of complex vector spaces $V_v$ for all vertices $v \in Q_0$ together with a collection $\pmb \varphi$ of linear maps $\varphi_a: V_{ta} \to V_{ha}$ for all $a\in Q_1$. For all but finitely many vertices $v$ the spaces $V_v$ are required to be zero. Morphisms $\mathbf{f: R \to R}$ between representations $\mathbf{ R =(V,\pmb \varphi)}$ and $\mathbf{ R' =(V',\pmb \varphi')}$ by definition consist of linear maps $f_v: V_v \to V'_v$ for all $v\in Q_0$ such that $\varphi_a \circ f_{ta} = f_{ha} \circ \varphi_a$. Given a representation $\mathbf{ R \,=\,(V,\pmb \varphi)}$, any non-trivial path $p$ in the sense of \eqref{eq:path} induces a linear map
$$
\varphi(p) \,:=\, \varphi_{a_0}\circ \cdots \circ \varphi_{a_m} : V_{tp}\to V_{hp}.
$$
The linear map $\varphi(e_v)$ that is induced by the trivial path at a vertex $v$ is by definition the identity $id: V_v \to V_v$.

A linear representation $\mathbf{ R =(V,\pmb \varphi)}$ is said to satisfy a relation $r= c_1 p_1 + \ldots c_\ell p_\ell$, if
$$
c_1 \varphi(p_1) + \ldots + c_\ell \varphi(p_\ell)=0.
$$
For any set $\cK$ of relations of a quiver $Q$ a {\em $(Q,\cK)$-module} is a linear representation of $Q$ that satisfies all relations from $\cK$.

\subsection{$P$-modules and quiver representations}

Let $G$ be a connected semisimple affine algebraic group defined over the complex
numbers. Let $P$ be a proper parabolic subgroup of $G$. In \cite[p.~8,
Section~1.3]{A-G03}, a quiver $Q$ with relations $\mathcal K$ is constructed from $P$.
We will adopt the notation of \cite[Section~1.3]{A-G03}. The reader is referred to
\cite[Section~1.3]{A-G03} for the details.

In \cite[pp.~8--9, Theorem~1.4]{A-G03} the following equivalence of categories is proved:
$$
\left\{\vcenter{\hbox{\text{finite dimensional holomorphic}}\hbox{\text{representations of $P$}}}\right\}\longleftrightarrow
\left\{\vcenter{\hbox{\text{finite dimensional representations}}\hbox{\text{of $Q$ satisfying relations in $\mathcal K$}}}\right\}\, .
$$

\subsection{Holomorphic quiver bundles and sheaves}

We assume that $Q$ is a locally finite quiver. In a natural way the notion of linear representations of quivers extends to the notion of $Q$-sheaves on complex manifolds.
\begin{definition}\label{de:Qsheaf}
A $Q$-sheaf $\pmb{{\cR}}=\pmb{(\cE,\phi)}$ on a compact complex manifold $X$ consists of a collection $\pmb\cE$ of coherent sheaves $\{\cE_v\}_{v\in Q_0}$ together with a collection $\pmb{\phi}=\{\phi_a\}_{a\in Q_1}$ of morphisms $\phi_a: \cE_{ta} \to \cE_{ha}$ such that $\cE_v=0$ for all but finitely many $v\in Q_0$.
\end{definition}
Given a $Q$-sheaf\; $\pmb{{\cR}}=\pmb{(\cE,\phi)}$ on $X$ every path $p$ induces a morphism of sheaves: Let $p= a_0 \cdots a_m$, then
$$
\phi_p:=\phi_{a_0}\circ \cdots \circ \phi_{a_m}: \cE_{tp} \to \cE_{hp}.
$$
The trivial path $e_v$ at a vertex $v$ induces the identity $\phi_v =id : \cE_v \to \cE_v$.

A $Q$-sheaf\; $\pmb{{\cR}}=\pmb{(\cE,\phi)}$ satisfies a relation $$r\,=\,c_1 p_1 + \cdots +
c_\ell p_\ell$$ if $c_1 \phi(p_1) + \cdots + c_\ell \phi(p_\ell)=0$. It is understood
that for all paths in a relation have a common initial point and a common
end point. Given a set
of relations $\cK$, a $Q$-sheaf with relations $\cK$ or a {\em $(Q,\cK)$-sheaf} is
defined like in the absolute case. Also the notion of morphisms of $Q$-sheaves is carried
over from the absolute case. If all sheaves $\cE_v$ are locally free, a $Q$-sheaf is
called a holomorphic $Q-$bundle.

It can be seen immediately that the category of $(Q,\cK)$-sheaves is abelian.

\subsection{Correspondence between equivariant sheaves, equivariant filtrations, and quiver
sheaves}

In the absolute case, where $X$ is a point, it is known (cf.\ \cite[Corollary 1.19]{A-G03}) that a holomorphic, homogeneous vector bundle $\cF$ on $G/P$ admits a filtration of holomorphic, homogeneous subbundles: Given an irreducible representation
$M_\lambda$ of $P$ corresponding to an integral dominant weight
$\lambda \in \Lambda^+_P$, then $$\cO_\lambda\,:=\, G\times_P M_\lambda$$ is the induced irreducible holomorphic homogeneous vector bundle on $G/P$. The filtration of $\cF$ is of the form
\begin{gather}
 {\ul\cF}: 0 \hookrightarrow \cF_0 \hookrightarrow \cF_1 \hookrightarrow \ldots \hookrightarrow \cF_m=\cF\\
 \cF_j/\cF_{s-1} \simeq V_{\lambda_j} \otimes \cO_{\lambda_j} \text{ for } s=1,\ldots, m
\nonumber
\end{gather}
for some dominant integral weights $\lambda_j$ with $\lambda_0 < \lambda_1 < \ldots < \lambda_m$ with $V_\lambda$ determined by $\cF$.

Now let $M \,=\, X \times (G/P)$ as above and $\cF$ a coherent $G$-equivariant sheaf on $M$. The a filtration
\begin{equation}\label{eq:filtM}
{\ul\cF}: 0 \hookrightarrow \cF_0 \hookrightarrow \cF_1 \hookrightarrow \ldots \hookrightarrow \cF_m=\cF
\end{equation}
by coherent $G$-equivariant subsheaves is called a sheaf filtration. In an analogous way filtrations for holomorphic vector bundles in terms of subbundles are defined. These are called {\em holomorphic filtrations}.

\begin{theorem}[{Álvarez-Consul, García-Prada \cite[Theorem 2.5]{A-G03}}]
There is an equivalence of categories
$$
\left\{\vcenter{\hbox{\text{coherent $G$-equivariant}}\hbox{\text{sheaves on $X\times (G/P)$}}}\right\}\longleftrightarrow \left\{\text{$(Q,\cK)$-sheaves}\right\}.
$$
The holomorphic $G$-equivariant vector bundles on $X\times (G/P)$ and the holomorphic $(Q,\cK)$-bundles on $X$ are in correspondence by this equivalence.
\end{theorem}

\subsection{Invariant holomorphic structures on G-manifolds}
We consider the compact complex manifold $M= X \times (G/P)$, and the maximal compact subgroup $K\subset G$ such that $K/J\,\longrightarrow\, (G/P)$ is a diffeomorphism, where $J=K\cap Q$. The compact \ka manifold $Y = G/P \simeq K/J$ will be equipped with a $K$-invariant \ka form $\omega_Y$.

Let $F$ be a smooth $K$-equivariant complex vector bundle on $M$. The group $G$ is
considered a complexification of $K$ which ensures that $K$-actions on $F$ lift in a
unique way to
$G$-actions implying that $F$ carries the structure of a $G$-equivariant bundle on the
$G$-manifold $M$. Like in the classical case complex structures can be described. Let
$\cD$ be the space of Dolbeault operators (also denoted by $\db$-operators
in the literature) on $F$, on which the group $G$ acts by
$$\gamma(\db_F)\,=\, \gamma\circ \db_F \circ \inv \gamma$$ (for $\gamma\in G$ and $\db_F\in
\cD$). This action leaves invariant the space $\cC$ of integrable
Dolbeault operators. The
group $G$ acts by conjugation on the complexified gauge group ${\cG}^\C= \cA^0(Aut(F))$
of $F$ -- the fixed points of the action correspond in a unique way to the holomorphic
structures on $F$ such that the action of $G$ is holomorphic.

Let $\cE$ be a holomorphic $P$-equivariant vector bundle on $X$; the underlying
complex $J$-equivariant vector bundle is denoted by $E$. We denote by $\cF$ the induced
$G$-equivariant holomorphic vector bundle $G\times_P\cE$ on $X\times (G/P)$. The
underlying complex vector bundle is $F\,=\,K\times_J E$ on the same base with respect to the
differentiable structure given by $$X \times (J/K) \,\simeq\, X\times (G/P)\, .$$

Concerning the underlying real structure, given an irreducible representation
$M_\lambda$ of $J$ corresponding to an integral dominant weight $\lambda$, the induced
irreducible smooth homogeneous vector bundle on $K/J$ is $H_\lambda\,=\,
K\times _J M_\lambda$. The Dolbeault operator corresponding to the holomorphic
structure $\cO_\lambda$ is denoted by $\db_{H_\lambda}$. The maps $p \,:\, X \times K/J
\,\to\, X$ and $q\,:\, X \times K/J \,\to\, K/J$ denote the canonical projections.

\begin{lemma}[{\cite[Lemma 3.1]{A-G03}}]\label{le:decompF}
Every smooth K-equivariant complex vector bundle $F$ on $X \times K/J$ can
be equivariantly decomposed, uniquely up to isomorphism, as
\begin{equation}\label{decf}
F \simeq \oplus_{\lambda \in Q'_0} F_\lambda\, , \quad F_\lambda:=
p^*E_\lambda \otimes q^*H_\lambda,
\end{equation}
for some finite collection $\mathbf E$ of smooth complex vector bundles $E_\lambda$ on $X$, with trivial $K$-action, where $Q'_0 \subset Q_0$ is the set of vertices with $E_ \lambda\neq 0$, which is known to be finite.
\end{lemma}

As in \cite[p.~26, Notation~3.2.2]{A-G03}, we take a basis $$\{\eta_a ; a \in Q_0, ta
= \lambda, ha = \mu\}$$ of $\cA^{0,1}(Hom(H_\lambda,H_\mu))^K$. The reader is referred
to \cite[p.~26, Lemma~3.3]{A-G03} and \cite[p.~26, Notation~3.2.2]{A-G03}.

In order to define the Dolbeault operators, we refer to the following notation. We invoke the situation of Lemma~\ref{le:decompF}. Again $\cD$ stands for the space of all Dolbeault operators on $F$ and $\cC$ is the subspace of integrable such operators. Let $Q'=(Q'_0,Q'_1)$ be the full sub-quiver of $Q$, whose vertices $\lambda$ are defined by the condition $E_\lambda\neq 0$, and whose arrows $a$ by definition satisfy $E_{ta}\,\neq\, 0$ and
$E_{ha}\,\neq\, 0$.

For $\lambda \in Q'_0$ there is a complex gauge group $\cG^\C_\lambda$ of $E_\lambda$ and
$$
\cG^\C = \prod_{\lambda \in Q'_0} \cG^\C_\lambda,
$$
which acts on the space $\cD'$, and on the representation space $\cR(Q',
\mathbf E)$, being defined by
$$
\cD' = \prod_{\lambda \in Q'_0} \cD_\lambda, \qquad \cR(Q', \mathbf E) = \bigoplus_{a\in Q'_1}\Omega^0(Hom(E_{ta}, E_{ha})).
$$
As the above elements of gauge groups decompose in the above way, they act by conjugation on both the Dolbeault operators from $\cD'$ and homomorphisms from $\cR(Q', \mathbf E)$ thus amounting to an action of $\cG^\C$ on $\cD' \times \cR(Q', \mathbf E)$. The integrability condition
$$\db^2\,=\,0$$ on a Dolbeault
operator $\db$
determines an invariant subspace $\cN=\{(\db_E,\pmb \phi )\}$, where the operators $\db_E$ define a complex structure on $E$, and where the $\phi_a$ are holomorphic. Furthermore the resulting $Q$-bundle $\pmb{{\cR}}=\pmb{(\cE,\phi)}$ is supposed to satisfy the relations from $\cK$.

Like in the classical situation there is a correspondence of connections on the
``real bundles'' and Dolbeault operators on the complexified bundle.
These contain the set of holomorphic structures. We will need a precise
description.

\begin{proposition}[{\cite[Proposition 3.4]{A-G03}}]\mbox{}
\begin{itemize}
\item[(a)]
There is a one-to-one correspondence between $\cD^G$ and $\cD' \times \cR(Q', \mathbf E)$ which, to any $(\db_E,\,\pmb \phi )\,\in\, \cR(Q',\, \mathbf E)$ associates, the Dolbeault operator $\db_F\in \cD^G$
given by
\begin{equation}
\db_F = \sum_{\lambda \in Q'_0} \db_{F_\lambda}\circ \pi_\lambda + \sum_{a\in Q'_1} \beta_a\circ \pi_{ta}\, ,
\end{equation}
where $\pi_\lambda$ is the projection of $F$ to $F_\lambda$ in \eqref{decf}
and $\db_{F_\lambda} = p^* \db_{E_\lambda} \otimes id + id \otimes \db_{H_\lambda}$
for $\lambda \in Q'_0$, while $$\beta_a := p^*\phi_a \otimes q^*\eta_a \in \cA^{(0,1)}(Hom(F_{ta}, F_{ha}))$$ for $a \in Q'_1$.
\item[(b)]
The previous correspondence restricts to a one-to-one correspondence between $\cC^G$ and $\cN$.
\item[(c)]
There is a one-to-one correspondence between $(\cG^\C)^G$ and $\cG^{'\C}$ which, to any $g \in \cG^{'\C}$ associates $g= \sum_{\lambda \in Q'_0} \wt g_\lambda \circ \pi_\lambda \in (\cG^\C)^G$ and $\cG^{'\C}$, with $\wt g_\lambda = p^* g_\lambda \in \Omega^0(Aut(p^* E_\lambda))$.
\item[(d)]
These correspondences are compatible with the actions of the groups of (c) on the sets of (a) and (b), hence there is a one-to-one correspondence between $\cC^G/(\cG^\C)^G$ and $\cN/\cG^{'\C}$.
\end{itemize}
\end{proposition}

\subsection{Stability, generalized quiver vortex equations, and
Donaldson-Uhlenbeck-Yau correspondence}

Again $G$ denotes a complex connected semisimple affine algebraic group with
$K\subset G$ being a maximal compact subgroup. We assume that $G$ acts on a
(compact) \ka manifold $(M,\omega)$ such that $K$ fixes the \ka form $\omega$.

Let $\cF$ be a locally free $G$-equivariant sheaf on $M$ and a filtration of the form \eqref{eq:filtM} be given, and let $h$ be a hermitian metric on $\cF$. Denote by $F_h$ the curvature form of $h$ and by $\Lambda$ the adjoint of the homomorphism given by wedge product with $\omega$. Then the generalized \he condition (or deformed \he condition) reads
\begin{equation}\label{eq:heMatrix}
\ii \Lambda F_h = \left(
\begin{array}{cccc}
\tau_0 I_0 & & & \\
& \tau_1 I_1 & & \\
& & \ddots & \\
& & & \tau_m I_m
\end{array}
\right)
\end{equation}
where the blocks of the matrix correspond to the filtration \eqref{eq:filtM}, and where the real numbers $\tau_j$ are subject to one condition that is determined by the filtration and the \ka form. More precisely, taking traces in \eqref{eq:heMatrix} and integrating over $M$ yields
\begin{equation}\label{eq:rkdeg}
\sum^{m}_{j=0} \tau_j \rk(\cF_j/\cF_{j-1})= \deg(\cF).
\end{equation}
The condition \eqref{eq:heMatrix} is called $\tau$-\he condition with $\tau=(\tau_0,\ldots,\tau_m)$. The group $K$ acts on the space of hermitian metrics as follows: If $\gamma\in K$ and $\wt h$ is a hermitian metric, then $\gamma \cdot \wt h = (\inv \gamma)^*\wt h$. We will be interested in $K$-invariant \he metrics for a holomorphic filtration.

There exists the corresponding notion of (semi)stability for equivariant holomorphic filtrations.
\begin{definition}
Let $\sigma=(\sigma_0,\ldots,\sigma_{m-1})\in \R^{m}$, and let ${\ul\cF}$ be a $G$-equivariant
holomorphic filtration. The $\sigma$-degree and $\sigma$-slope are defined by
$$
\deg_\sigma (\cF) = \deg(\cF) + \sum^{m-1}_{j=0} \sigma_j \rk(\cF_j), \quad \mu_\sigma(\cF) = \frac{\deg_\sigma(\cF)}{\rk(\cF)}.
$$
\end{definition}
The $G$-equivariant filtration $\cF$ is called $\sigma$-stable (respectively,
$\sigma$-semistable)
if for all $G$-invariant proper sheaf subfiltrations $\cF' \,\hookrightarrow\, \cF$ the
$\sigma$-slopes satisfy $\mu_\sigma(\cF') \,<\, \mu_\sigma(\cF)$
(respectively, $\mu_\sigma(\cF') \,\leq\, \mu_\sigma(\cF)$). Note that a subfiltration by coherent sheaves is given by a coherent $G$-equivariant subsheaf $\cF'$, which generates a filtration by cutting down the given one. A $\sigma$-polystable filtration is by definition the direct sum of ($G$-equivariant) $\sigma$-stable filtrations, which have the same slope $\sigma$.

The numbers $\sigma_j$ and $\tau_j$ will be related by
\begin{equation}\label{eq:st}
\sigma_j = \tau_{j+1} - \tau_j, \;\text{ and }\; \sigma_j>0\; \text{ for } j=0,\ldots, m-1.
\end{equation}

\begin{theorem}[{\cite[Theorem 4.7]{A-G03}}]
Let $\cF$ be a $G$-equivariant holomorphic filtration on a compact \ka $K$-manifold $M$. Let $\tau$ and $\sigma$ satisfy \eqref{eq:rkdeg}, and \eqref{eq:st}. Then $\cF$
admits a $K$-invariant $\tau$-Hermite-Einstein metric, if and only if it is
$\sigma$-polystable (in the $G$-equivariant sense).
\end{theorem}

The \kh correspondence for equivariant filtrations, together with the \he equation
\eqref{eq:heMatrix} transfers to the holomorphic quiver bundle as follows. Note
that $M= X \times G/P$ carries the \ka form induced by the \ka form $\omega_X$ on
$X$ and the $K$-invariant \ka form $\omega_Y$ on $G/P$.

\begin{theorem}[{\cite[Theorem 4.13]{A-G03}}]\label{th:A-G}
Let $\cF$ be a $G$-equivariant holomorphic vector bundle on $X\times G/P$ with associated filtration $\ul \cF $, and let $\pmb{{\cR}}\,=\,\pmb{(\cE,\phi)}$ be its corresponding holomorphic $(Q,\cK)$-bundle on $X$, where $(Q,\,\cK)$ is the quiver with relations associated to $P$. Then $\ul\cF$ has a $K$-invariant $\tau$-\he metric, with respect to the \ka form $\omega=p^*\omega_X + q^*\omega_Y$, if and only if the vector bundles $\cE_\lambda$ in $\pmb\cR$ admit hermitian metrics $k_\lambda$ on $\cE_\lambda$ for each $\lambda\in Q_0$ with $\cE_\lambda\neq0$, satisfying
\begin{equation}\label{eq:HE}
\ii n_\lambda \Lambda F_{k_\lambda} + \sum_{a\in \inv h(\lambda)} \phi_a \circ \phi^*_a - \sum_{a\in \inv t(\lambda)} \phi^*_a \circ \phi_a = \tau'_\lambda \mathrm{id}_{\cE_\lambda},
\end{equation}
where $F_{k_\lambda}$ is the curvature form of the metric $k_\lambda$ on the
holomorphic vector bundle $\cE_\lambda$, for each $\lambda \,\in\, Q_0$ with $\cE_\lambda
\,\neq\,0$, and $n_\lambda\, =\, \dim N_\lambda$ is the multiplicity of the irreducible representation $M_\lambda$, for $\lambda\,\in\, Q_0$.
\end{theorem}

The above real numbers $\tau'_\lambda$ are related to the numbers $\tau_\lambda$ according to \cite[4.2.2]{A-G03}.

\section{Deformation theory}\label{se:def}

Let $M$ be a compact Kähler manifold, and $\cF$ a locally free, coherent sheaf on $M$ together with a filtration $\ul\cF$ in the sense of \eqref{eq:filtM}. A holomorphic family of such objects parameterized by a complex space $S$ is a collection of filtrations
$\{{\ul\cF}_s\}_{s\in S}$, where each filtration $\ul\cF_s$ is the restriction of a
filtration $\ul\cF$ over the space $M\times S$ to $M\times \{s\}\,\simeq\, M$. For $M
\,=\,X \times G/P$ we require that the group $G$ acts fiberwise in the above sense. A {\em deformation} of $\ul\cF$ is defined over a complex space $S$ with a distinguished base point $s_0\in S$. It consists of a holomorphic family over $S$ together with an isomorphism of the given object and the fiber of $s_0$. Since deformation theory only classifies holomorphic families in a local way with respect to the base, the base spaces may have to be replaced by neighborhoods of the distinguished base point.

It is known that {\em stable} quiver sheaves are {\em simple}, i.e.,\ all endomorphisms are constant multiples of the identity (cf.\ \cite{A-G03b}).

In principle, the deformation theory of such objects is well understood. However,
the quiver representation is the method of choice for a differential geometric
study of the moduli spaces. It should be mentioned that Section \ref{ss:infdef}
and Section \ref{sec3.2} remain valid for general quiver bundles associated to
locally finite quivers (not necessarily associated to $G/P$). However from Section \ref{se3.3}
the discussion is restricted to quiver bundles equipped with special metrics.

We will consider holomorphic families of quiver bundles with relations associated to $P$
and $G$-equivariant holomorphic vector bundles on $X \times G/P$ with associated
filtration $\ul\cF$. Without loss of generality, when using methods of deformation
theory, the base spaces will by assumption be Stein and contractible, in order to avoid
unnecessary complications. Furthermore, for applications to moduli spaces, the base
spaces will also be reduced.

A crucial role will be played by Hermitian metrics that satisfy the equations of vortex type
according to Theorem~\ref{th:A-G}. We will call these {\em ``quiver vortex equations''}.

\subsection{Deformations of quiver bundles}\label{ss:infdef}

The tangent cohomology for the deformation problem for holomorphic quiver bundles is a
certain hypercohomology.

Let $\pmb{{\cR}}=\pmb{(\cE,\phi)}$ be holomorphic quiver bundle on $X$ satisfying
equations from a set $\cK$ so that it carries the structure of a $(Q,\cK)$-bundle. We use
the following notation. Again, holomorphic families are well defined. Note that the
combinatorial data, in particular the set of equations $\cK$ originate from linear
representations of quivers, which depend on $P\subset G$ and are constant in a
holomorphic family.

We will use the following {\em notation}. Again we will deal with quiver bundles satisfying a set of equations $\cK$. We denote by
$\pmb{{\wt\cR}}=\pmb{(\wt\cE,\wt\phi)}$ a holomorphic quiver bundle over $X\times S$,
where $S$ denotes a parameter space, the fibers of a a point $s\in S$ i.e.\ the
restrictions to $X \times \{s\} \simeq X$ will be denoted by
$\pmb{{\wt\cR}_s}=\pmb{(\wt\cE_s,\wt\phi_s)}$. Furthermore an isomorphism
$\pmb{{\cR}}\stackrel{\sim}{\to}\pmb{{\wt\cR}_{s_0}}$ is being fixed.

\subsection{Tangent cohomology for holomorphic quiver bundles}\label{sec3.2}

Let
\begin{gather*}
{\Delta} \,:\, \oplus _{\lambda\in Q'_0} \End(\cE_\lambda) \,\longrightarrow\,
\oplus _{a\in Q'_1} \Hom(\cE_{ta},\cE_{ha})\\ \Delta(\xi)\,:=\,
[\xi,\phi], \text{ with } [\xi,\phi]_a \,:=\,
 \frac{1}{\sqrt{n_{ha}}} \xi_{ha}\circ \phi_a -\frac{1}{\sqrt{n_{ta}}}\phi_a \circ \xi_{ta}.
\end{gather*}
Let $\psi = \Delta(\xi)$. By assumption $r(\phi)=0$ for all equations $r$ from $\cK$. It follows immediately that $r(\psi)=0$ so that $\Delta$ is compatible with $\cK$.

Let
\begin{eqnarray*}
A^1&=&\big\{\psi =(\psi_a)_{a\in Q'_1};\psi_a \in \Hom(\cE_{ta},\cE_{ha}), r(\psi)=0 \text{ for all } r\in \cK \big\},\\
B^0&=& \big\{\xi=(\xi_\lambda)_{\lambda\in Q'_0}; \xi_\lambda\in \End(\cE_\lambda), r(\Delta(\xi_\lambda))=0 \text{ for all } r\in \cK \big\}.
\end{eqnarray*}
Now we define three complexes
\begin{eqnarray*}
A^\bullet &:& 0 \longrightarrow 0 \longrightarrow A^1 \longrightarrow 0\\
B^\bullet &:& 0 \longrightarrow B^0 \longrightarrow 0 \longrightarrow 0\\
C^\bullet &:& 0 \longrightarrow B^0 \stackrel{\Delta}{\longrightarrow} A^1 \longrightarrow 0.
\end{eqnarray*}
We compute hypercohomologies from $0 \to A^\bullet \to C^\bullet \to B^\bullet \to 0$ and use $\mathbb H^q(A^\bullet) = H^{q-1}(A^1)$ and $\mathbb H^q(B^\bullet)= H^q(B^0)$. Since $\Delta$ descends to the cohomology, there is an exact sequence
\begin{equation}\label{eq:longhyper}
0 \to \mathbb H^0(C^\bullet) \to H^0(B^0) \stackrel{\Delta}{\longrightarrow} H^0(A^1) \to \mathbb H^1(C^\bullet) \to H^1(B^0) \stackrel{\Delta}{\longrightarrow} H^1(A^1) \to \mathbb H^2(C^\bullet)\to \ldots
\end{equation}
\begin{proposition}
The hypercohomology of $C^\bullet$ is the tangent cohomology for deformations of holomorphic quiver bundles satisfying the given set of equations $\cK$.
\begin{itemize}
\item[(i)]
The group $\mathbb H^0(C^\bullet)$ can be identified with the space of infinitesimal automorphisms of $\pmb{{\cR}}=\pmb{(\cE,\phi)}$.
\item[(ii)]
The group $\mathbb H^1(C^\bullet)$ can be identified with the space of infinitesimal deformations of \break $\pmb{{\cR}}=\pmb{(\cE,\phi)}$.
\end{itemize}
\end{proposition}
We mention that $\mathbb H^2(C^\bullet)$ contains the obstructions.
\begin{proof}
We use \eqref{eq:longhyper} for a direct argument. Given any family of endomorphisms $(\chi_\lambda) \in B^0$, we use the notation
$$
\chi'=(\chi'_\lambda)_{\lambda\in Q'_0}, \text{ where } \chi'_\lambda = \frac{1}{\sqrt{n_\lambda}}\chi_\lambda.
$$
Now the condition $\Delta(\chi)=0$ exactly means that $\chi'$ defines an endomorphism of the holomorphic quiver bundle, which proves (i).

In order to prove (ii) we are looking at
$$
0\to H^0(A^1)\big/ {\rm Im} \left(H^0(B^0) \stackrel{\Delta}{\longrightarrow} H^0(A^1)\right) \stackrel{\sigma}{\longrightarrow} \mathbb H^1(C^\bullet) \stackrel{\tau}{\longrightarrow} {\rm Ker}\left( H^1(B^0)\stackrel{\Delta}{\longrightarrow} H^1(A^1)\right) \to 0.
$$

The elements of $H^1(B^0)$ stand for the infinitesimal deformations of the bundles $\cE_\lambda$, namely equivalence classes of extensions of the sheaves $\cE_\lambda$ (considered as independent of each other). Let $(0,\cO_D)$ be the double point, where $\cO_D = \C \oplus \epsilon \C$ with $\epsilon^2=0$. Then an extension $\ul \cE_\lambda$ of a sheaf $\cE_\lambda$ by itself, possesses a natural $\cO_D$-algebra structure. We write
$$
\ul \cE_\lambda : 0 \to \epsilon \cE_\lambda \to \wt \cE_\lambda \to \cE_\lambda \to 0.
$$
for the extension, where $\wt \cE_\lambda$ is equipped with the $\cO_D$-algebra structure. With respect to an open covering $\{U_i\}$ of the underlying manifold we introduce for all extensions $\ul\cE_\lambda$ initializations
$$
\ul \cE_\lambda|U_i : 0 \to \epsilon \cE_\lambda|U_i \,\to \,\cE_\lambda|U_i \oplus \epsilon\cE_\lambda|U_i \,\to \cE_\lambda|U_i \to 0
$$
with identifications over $U_{ij}=U_i \cap U_j$ being given by a cocycle $(\xi'_{ij}) = (\xi'_{ij,\lambda})_{\lambda\in Q'_0}$, $\xi'_{ij,\lambda}\in End(U_{ij},\cE_\lambda)$, which maps $\cE_\lambda|U_{ij}$ to $\epsilon \cE_\lambda|U_{ij}$. Let $\xi_{ij,\lambda}= \sqrt{n_\lambda}\, \xi'_{ij,\lambda}$.

We also have the homomorphisms $\phi_a$:
$$
\xymatrix{
\ul\cE_{ta}|U_i : & 0 \ar[r] & \epsilon \cE_{ta}|U_i \ar[r] \ar[d]_{\epsilon\phi_a|U_i} &\cE_{ta}|U_i \oplus \epsilon \cE_{ta}|U_i \ar[d]_{{\phi_a|U_i} \oplus \epsilon {\phi_a|U_i}}\ar[r] & \cE_{ta}|U_i \ar[r]\ar[d]_{\phi_a|U_i} & 0
\\
\ul\cE_{ha}|U_i : & 0 \ar[r]& \epsilon \cE_{ha}|U_i \ar[r] &\cE_{ha}|U_i \oplus \epsilon \cE_{ha}|U_i \ar[r] &\cE_{ha}|U_i \ar[r]& 0
}
$$

We assume that $(\xi_{ij,\lambda})_{\lambda\in Q'_0}$ represents an element of the kernel of $\Delta: H^1(B^0) \to H^1(A^1)$, i.e.\ $\Delta(\xi_{ij})$ is a coboundary in $A^1$ given by a $0$-cochain $(\theta_{i,\lambda})_{\lambda\in Q'_0}$ in $A^0$. The $\theta'_{i,\lambda}= (1/\sqrt{n_\lambda}) \, \theta_{i,\lambda}$ consist of endomorphisms or $\cE_\lambda|U_i$. These define a change of the trivializations of the $\cE_\lambda|U_i$.

After applying these, we can assume that $\Delta(\xi)=0$ or equivalently
$$
\xi'_{ij,ha} \circ(\phi_a|U_{ij})= (\phi_a|U_{ij})\circ \xi'_{ij, ta}
$$
hold.

These imply that the identification maps are compatible with $\pmb\phi $. Altogether we have extensions satisfying
$$
\wt \cE_{ta} \stackrel{\wt\phi_a}{\longrightarrow} \wt \cE_{ha}
$$
for all $a \in Q'_1$ over the double point that restrict to the given holomorphic quiver bundle, and we saw that the image of $\tau$ consist of all infinitesimal deformations of the sheaves $\cE_\lambda$ arising from deformations of the given quiver.

We are left with showing that the image of $\sigma$ corresponds to all
infinitesimal deformations of the given quiver that induce trivial deformations of
all $\cE_\lambda$.

More generally, we see that given an infinitesimal deformation $(\wt \cE_\lambda,
\wt \phi_a)$ the homomorphisms $\wt \phi_a$ are determined by the deformations
$\wt\cE_\lambda$ only unique up to adding homomorphisms of the form
$\beta\circ\chi_a\circ\nu$, with $\chi \in H^0(A^1)$. The images under $\sigma$ are given by $\chi_a\,:\,\cE_{ta} \,\to\,
\epsilon\cE_{ha}$ in the following way:
$$
\xymatrix{
\ul\cE_{ta} : & 0 \ar[r] & \epsilon \cE_{ta} \ar[r]^\alpha \ar[d]_{\epsilon\phi_a } &\wt\cE_{ta} \ar[d]^{\phantom{M}\wt\phi_a } \ar[r]^\nu & \cE_{ta} \ar[r]\ar[d]_{\phi_a } \ar[dll]|{_{\chi_a\phantom{MMMMMM}}} & 0
\\ \ul\cE_{ha} : & 0 \ar[r] &\epsilon\cE_{ha} \ar[r]^\beta&\wt\cE_{ha}\ar[r]^\mu & \cE_{ha} \ar[r]& 0
}
$$
This fact accounts for the image of $H^0(A^1)$ in $\mathbb H^1(C^\bullet)$.
The only homomorphisms that act ineffectively, are those which are induced by isomorphisms of the single
extensions $\ul\cE_\lambda$. These constitute the space $\Delta(H^0(B^0))$: Let $\varphi\,=\,(\varphi_\lambda)_{\lambda\in Q'_0}$
with $\varphi'_\lambda \,= \,(1/\sqrt{n_\lambda})\varphi_\lambda\,:\,{\mathcal E}_\lambda\,\to\,
\epsilon {\mathcal E}_\lambda$. Then $\Delta(\varphi)$ gives rise to the diagram
$$
\xymatrix{
\ul\cE_{ta} : & 0 \ar[r] & \epsilon \cE_{ta} \ar[r]^\alpha \ar@{=}[d] &\wt\cE_{ta} \ar[d] \ar[r]^\nu & \cE_{ta} \ar[r]\ar@{=}[d] \ar[dll]|{\phantom{MMMM}\varphi'_{ta}} & 0
\\
\ul\cE_{ta} : & 0 \ar[r] & \epsilon \cE_{ta} \ar[r]^\alpha \ar[d]_{\epsilon\phi_a } &\wt\cE_{ta} \ar[d]^{\phantom{}\wt\phi_a } \ar[r]^\nu & \cE_{ta} \ar[r]\ar[d]_{\phi_a } & 0
\\
\ul\cE_{ha} : & 0 \ar[r] &\epsilon\cE_{ha} \ar[r]^\beta \ar@{=}[d] &\wt\cE_{ha}\ar[r]^\mu \ar[d] & \cE_{ha} \ar[r]\ar@{=}[d]\ar[dll]|{\phantom{MMMM}\varphi'_{ha}}& 0\\
\ul\cE_{ha} : & 0 \ar[r] &\epsilon\cE_{ha} \ar[r]^\beta&\wt\cE_{ha}\ar[r]^\mu & \cE_{ha} \ar[r]& 0
}
$$
replacing $\wt \phi_a$ by $\wt\phi_a + \wt\phi_a\circ\alpha\circ\varphi'_{ta}\circ\nu - \beta\circ\varphi'_{ha}\circ\mu\circ\wt \phi_a $.
\end{proof}

\subsection{Families of quiver vortex structures, application to deformation theory}\label{se3.3}

In \linebreak \cite{A-G03b} the simplicity of stable quiver bundles was shown, a fact that implies the following statement (cf.\ \cite{f-s}, \cite[Theorem 4.3]{B-S09}).

\begin{proposition}\label{pr:impfct}
Given a stable holomorphic quiver bundle $\pmb{{\cR}}=\pmb{(\cE,\phi)}$ equipped with a quiver vortex structures, and a deformation over a space $(S,s_0)$, then the quiver vortex structure can be extended to a family of such structures over a neighborhood of $s_0$ in $S$. The hermitian quiver vortex metrics depend in a $\cinf$ way upon the parameter, whereas the maps $\phi_\lambda$ depend holomorphically upon the parameter $s\in S$.
\end{proposition}

Given a holomorphic quiver bundle $\pmb{{\wt\cR}}=\pmb{(\wt\cE,\wt\phi)}$ over $X\times S$, equipped with a quiver vortex structure, the corresponding hermitian metrics on the holomorphic vector bundles $\wt \cE_\lambda$ will be denoted by $k^\lambda$ with curvature tensor $R^\lambda$ over $X\times S$.

Our aim us to introduce an acyclic resolution $C^\bullet \to C^{\bullet\bullet}$, from which the hypercohomology can be computed. A Dolbeault complex serves this purpose.

We consider the induced sheaf analogues $\cA^1$, $\cB^0$, $\cC^\bullet$ of the corresponding spaces $A^1$, $B^0$, and $C^\bullet$ from Section~\ref{sec3.2}, and the related Dolbeault resolutions by differential forms that satisfy the quiver relations from $\cK$. The latter property is indicated by a subscript $\cK$.
\begin{eqnarray*}
0\to \cA^1 &\to& \cA^{(0,\bullet)}(( \oplus _{a\in Q'_1} \Hom(\cE_{ta},\cE_{ha}))_\cK)\\
0 \to \cB^0 & \to & \cA^{(0,\bullet)}((\oplus _{\lambda\in Q'_0}\End(\cE_\lambda))_\cK)
\end{eqnarray*}
These fit into a resolution of the complex $\cC^\bullet$ of sheaves. We will compute the hypercohomology of
$$
\cC^\bullet: 0 \to \cB^0 \stackrel{\Delta}{\longrightarrow} \cA^1 \to 0 \to 0 \ldots
$$ from the induced double complex of the Dolbeault resolutions.

\begin{lemma}
A resolution $\cC^\bullet \to \cC^{\bullet\bullet}$ is defined by
$$
\cC^{p,q}=
\begin{cases}
\cA^{(0,q)}((\oplus _{\lambda\in Q'_0}\End(\cE_\lambda))_\cK) & \text{ for } p=0\\ \cA^{(0,q)}(( \oplus _{a\in Q'_1} \Hom(\cE_{ta},\cE_{ha}))_\cK)
& \text{ for } p=1\\
0 & \text{ for } p\neq 0,1
\end{cases}
$$
with boundary operators $\db: \cC^{p,q} \to \cC^{p,q+1}$ for $p=0,1$, $q\geq0$ and $(-1)^q \Delta : \cC^{0,q} \to \cC^{1,q}$ for $q\geq 0$.
\end{lemma}
\begin{proof}
The operators $\db$ and $\Delta$ commute because the morphisms $\phi_a$ are holomorphic.
Furthermore the Dolbeault complexes provide resolutions of both $\cC^0$ and $\cC^1$.
\end{proof}

\begin{lemma}\label{le:ellcompl}
Let $C^{\bullet\bullet} = \Gamma(X, \cC^{\bullet\bullet})$. Then the hypercohomology $\mathbb H^q(C^\bullet)$ is equal to the cohomology $H^q(C^{\bullet\bullet})$ of the double complex. The associated single complex is equal to an {\em elliptic complex} $\wt C^\bullet$ with
\begin{equation}\label{eq:simplcx}
\wt C^\bullet : 0 \to C^{0,0} \stackrel{d^{\, 0}}{\longrightarrow} C^{1,0} \oplus C^{0,1} \stackrel{d^1}{\longrightarrow} C^{1,1} \oplus C^{0,2} \stackrel{d^2}{\longrightarrow} C^{1,2} \oplus C^{0,3} \longrightarrow \ldots
\end{equation}
and
\begin{eqnarray}
d^{\, 0}\left((\xi_\lambda)_{\lambda \in Q'_0}\right) &=& \left((\Delta(\xi_\lambda))_a, (\db \xi_\lambda)_\lambda \right) \\
d^1\left((\chi_a)_{a \in Q'_1},(\zeta_\lambda)_{\lambda\in A'_0}\right) &=& \left((\db \chi_a)_a - (\Delta(\zeta_\lambda))_a , (\db( \zeta_\lambda))_\lambda\right)
\end{eqnarray}
\end{lemma}

We will need the following adjoint operators.

\begin{lemma}\label{le:adj}
The adjoint operators $(d^q)^*: \wt C^{q+1} \to \wt C^q$, $q= 0,1$ for \eqref{eq:simplcx}
are given by
$$
(d^{\, 0})^*\left( (\psi_a)_{a \in Q'_1} , (\zeta_\lambda)_{\lambda\in Q'_0} \right) =
\left(\frac{1}{\sqrt{n_\lambda}}\left(\sum_{a\in \inv h(\lambda)} \psi_a \phi^*_a - \sum_{a\in\inv t(\lambda)} \phi^*_a \psi_a\right) + \db^*\!( \zeta_\lambda) \right)_{\lambda\in Q'_0}
$$
$$
(d^1)^*\big( (\chi_a)_{a\in Q'_1}, (\xi_\lambda)_{\lambda\in Q'_0} \big)=\left( \left(\db^* \chi_a\right)_{\!a},\frac{1}{\sqrt{n_\lambda}} \left(\sum_{a\in \inv t(\lambda)} \phi^*_a \chi_a - \sum_{a\in \inv h(\lambda)} \chi_a \phi^*_a \right)+
\left(\db^*\xi_\lambda\right)_{\!\!\lambda}\right).
$$
\end{lemma}

Now we are in a position to compute tangent cohomology. We use the notation from
Section~\ref{ss:infdef}. Furthermore $(z^1,\ldots,z^n)$ are local coordinates on
the complex \ka manifold $(X,\omega_X)$. We use the semi-colon notation for
covariant derivatives. Note that the parameter space carries a flat structure. At
some point we will need local coordinates $(z^1,\ldots,z^n)$ on $X$, and we will
write $$\omega_X \,=\, \ii g_{\alpha\ol\beta}(z) dz^\alpha\we dz^\ol\beta\, .$$ The volume
element will be denoted by $g(z) dV$.

\begin{remark*} For simplicity, we use the letter $\phi_a$ also for homomorphisms $\wt\phi_a$ that depend on the holomorphic parameter.
\end{remark*}
\begin{proposition}\label{pr:mui}
Let a deformation of a holomorphic quiver bundle $\pmb{{\cR}}=\pmb{(\cE,\phi)}$ be given over a space $(S,s_0)$. Let $S\subset U \subset \C^k=\{s=(s^1,\ldots,s^k)\}$ be an embedding with $\mathrm{embdim}(S,s_0)= k$. Let
$$
\rho_{s_0} : T_{s_0}S \to \mathbb H^1(C^\bullet)
$$
be the \ks mapping. Then
\begin{equation}\label{eq:mui}
\mu_i= \left(- (\wt\phi_{a;i})_a , (\sqrt{n_\lambda}R^\lambda_{i\ol \beta}dz^\ol \beta)_\lambda \right) \in C^{1,0}\oplus C^{0,1}
\end{equation}
represents $\rho(\pt/\pt s^i)$. The tensor $\mu_i$ is harmonic:
$$
d^1(\mu_i) =0 \text{ and } (d^{\, 0})^*(\mu_i) =0.
$$
\end{proposition}
The $d^*$-closedness of $\mu_i$ can be interpreted as infinitesimal quiver vortex condition.
\begin{proof}
For hermitian metrics on families of holomorphic vector bundles the curvature component is known to represent the infinitesimal deformation (\cite{S-T-92}). Now $$
d^1(\mu_i)=\left( \db( \phi_{a;i})_a - (\Delta(\sqrt{n_\lambda}R^\lambda_{i \ol \beta} dz^{\ol\beta}))_a, \db(\sqrt{n_\lambda} R^\lambda_{i\ol\beta}dz^{\ol\beta})_\lambda\right).
$$
The curvature forms are obviously $\db$-closed, and the first components of $d^1(\mu_i)$ are given by
$$
\phi_{a;i\ol\beta} -R^{ha}_{i\ol\beta}\phi_a+ \phi_a R^{ta}_{i\ol\beta} = \phi_{a;\ol\beta i} =0.
$$
Next
$$
\left((d^{\, 0})^*(\mu_i)\right)_\lambda = \left( \db^*(\sqrt{n_\lambda}R^\lambda_{i\ol\beta}dz^{\ol\beta}) + \sum_{a\in \inv h(\lambda)} \frac{1}{\sqrt{n_\lambda}}\phi_{a;i}\phi^*_a -\sum_{a\in \inv{t}(\lambda)}\frac{1}{\sqrt{n_\lambda}}\phi_{a;i}\phi^*_a \right)_{\!\!\lambda}.
$$
Taking the derivative with respect to $s^i$ of the quiver vortex equation \eqref{eq:HE} shows that expression vanishes.
\end{proof}

\subsection{Generalized \wp metric for holomorphic quiver bundles}

We define a natural Hermitian structure in the tangent spaces $T_sS$ that is
compatible with base change morphisms. We call it $L^2$-metric or (generalized)
{\em \wp metric}. We use the notation $\pt_i$ for $\pt/ \pt s^i$.

\begin{definition}\label{dwp1}
A natural Hermitian structure $G^{WP}$ on $T_{s_0}S$ is given by
\begin{gather}\label{eq:wp}
G^{WP}_{i \ol \jmath} := G^{WP}(\pt_i,\pt_j) := \langle \mu_i,\mu_j\rangle := \int_X \tr(\mu_i \mu^*_{\ol\jmath})\\ \hspace{4cm}
:= \int_X \sum_{a \in Q'_1}\tr( \phi_{a;i} \phi^*_{a;\ol\jmath} ) gdV + \int_X \sum_{\lambda\in Q'_0}n_\lambda \tr(g^{\ol\beta\alpha} R^\lambda_{i \ol\beta}R^\lambda_{\alpha\ol\jmath}) gdV \nonumber
\end{gather}

We set
$$
\omega_{WP}= \ii G^{WP}_{i \ol\jmath} ds^i\we ds^\ol\jmath.
$$
\end{definition}

\section{Fiber integral formula and \ka property of the $L^2$-metric}

The \ka property for $\omega_{WP}$ can be derived from \eqref{eq:wp} using partial integration. In view of further applications we prove a fiber integral formula for $\omega_{WP}$.

We will use the technique of fiber integration for differentiable forms, which corresponds to the push-forward of currents. Let $n=\dim_\C X$. Given a differential form $\eta$ on $X\times S$ of degree $k + 2n$, the fiber integral
$$
\int_X \eta = \int_{X\times S/S} \eta
$$
is a differential form of degree $k$. Taking fiber integrals is type preserving, commutes with taking exterior derivatives $d$, $\pt$, and $\ol\pt$. In particular the fiber integral of a closed form is closed, and the integral of a form that possesses local $\pt\db$-potentials possesses local $\pt\db$-potentials. In our situation the base space $S$ need not be smooth, however the given $(n+1,n+1)$-forms on $X\times S$ will possess local extensions to smooth ambient spaces as $\pt\db$-exact forms of class $\cinf$ so that the fiber integral produces a $(1,1)$-form on the base, which possesses a $\pt\db$-potential on a smooth ambient space.

We set
$$
\eta^k=\frac{1}{k!} \eta\we\ldots\we\eta
$$
for any differential form $\eta$. (The notion of $c^2_1(\cE)$ is not being altered.)

\begin{proposition}\label{propfb1}
Denote by $\Omega^\lambda=R^\lambda_{\alpha\ol\beta}dz^\alpha\we dz^\ol\beta$ with $\lambda\in Q'_0$ the curvature form of $(\wt\cE_\lambda,k^\lambda)$. Then
\begin{gather*}
\hspace{-5cm}\omega_{WP}=\strut\frac{1}{2} \sum_{\lambda\in Q'_0}\left(
\int_X n_\lambda \tr(\Omega^\lambda\we\Omega^\lambda)\we \omega^{n-1}_X \right)+ \\
\strut \hspace{3cm} \ii \sum_{\lambda\in Q'_0}\left( \tau'_\lambda\int_X
\tr(\Omega^\lambda)\we \omega^n_X \right) +\sum_{a\in Q'_1}\idb
\left(\int_X \tr (\phi_a \phi^*_a) \omega^n_X\right)\, ,
\end{gather*}
where $\tau'_\lambda$ are as in \eqref{eq:HE}.
\end{proposition}

\begin{proof}
\begin{gather*}
\frac{1}{2}\sum_{\lambda\in Q'_0}\int_X
n_\lambda\tr(\Omega^\lambda\we\Omega^\lambda)\we \omega^{n-1}_X = -
\frac{1}{2}\sum_{\lambda\in Q'_0}\int_X
n_\lambda\tr(\ii\Omega^\lambda\we\ii\Omega^\lambda) \omega^{n-1}_X = \\
\ii\left(\sum_{\lambda\in Q'_0}\int_X
n_\lambda(R^\lambda_{\alpha\ol\jmath}R^\lambda_{i\ol\beta}
-R^\lambda_{\alpha\ol\beta}R^\lambda_{i\ol\jmath} )g^{\ol\beta\alpha} g dV\right)
ds^i\we ds^\ol\jmath=\\
\ii\left(\sum_{\lambda\in Q'_0}\int_X
n_\lambda(R^\lambda_{\alpha\ol\jmath}R^\lambda_{i\ol\beta} )g^{\ol\beta\alpha} g
dV\right) ds^i\we ds^\ol\jmath +
\ii \Bigg(\sum_{\lambda\in Q'_0} \int_X \Bigg(\sum_{a\in\inv h(\lambda)}\tr(
\phi_a\phi^*_a R^\lambda_{i\ol\jmath})\\
-\sum_{a\in\inv t(\lambda)} \tr(\phi^*_a\phi_a
R^\lambda_{i\ol\jmath}) \Bigg)gdV
-\sum_{\lambda\in Q'_0} \tau'_\lambda \int_X \tr(R^\lambda_{i \ol \jmath} ) gdV
\Bigg)ds^i\we ds^\ol\jmath
\end{gather*}
Now
$$
\tr\left(\phi_a\phi^*_a R^{ha}_{i\ol\jmath} -\phi^*_a\phi_a R^{ta}_{i\ol\jmath}\right)=\tr\left(\phi^*_a(\phi_{a;\ol\jmath i} - \phi_{a;i\ol\jmath} )\right) = -\tr(\phi^*_a \phi_{a;i\ol\jmath})=\tr((\phi^*_{a;\ol\jmath}\phi_{a;i})) - \left(\tr(\phi^*_a \phi_a)\right)_{i\ol\jmath}.
$$
This proves the proposition.
\end{proof}

\begin{corollary}\label{corfb1}
The \wp form $\omega_{WP}$ possesses locally a \ka potential. In particular, $\omega_{WP}$ is a \ka form.
\end{corollary}
We express the \wp form in terms of a fiber integral of Chern character forms.

\begin{proposition}\label{propfb2}
\begin{equation}
\frac{1}{4\pi^2} \omega_{WP}\,=
-\sum_{\lambda\in Q'_0}n_\lambda\int_X \mathrm{ch_2}\left( (\cE_\lambda, k^\lambda)\right)\we \omega^{n-1}_X
\end{equation}
$$
+ \sum_{\lambda\in Q'_0}\frac{\tau'_\lambda}{2\pi }\int_X\mathrm{c_1}(\cE_\lambda, k^\lambda )\we \omega^n_X + \pt \db \left( \frac{\ii}{4\pi^2}\sum_{a\in Q'_1}\int_X \tr(\phi_a \phi^*_a) \omega^n_X \right).
$$
\end{proposition}
Observe that $ \sum_{\lambda\in Q'_0}\mathrm{ch_2}\left( (\cE_\lambda, k^\lambda) \right)= \mathrm{ch}_2\left( \oplus_{\lambda\in Q'_0}(\cE_\lambda, k^\lambda)\right)$.

\begin{proof}
For a general holomorphic hermitian bundle $(\cE,h)$
\begin{equation}\label{eq:chchar}
\mathrm{ch}(\cE,h)= \sum^n_{k=0}\left(\frac{\ii}{2\pi}\right)^{\!k} \tr\left(\Omega(\cE,h)\right)^k
\end{equation}
holds (in terms of the convention on powers of differential forms). This proves the
proposition.
\end{proof}

\section{Determinant line bundle and Quillen metric}

We are given a proper, smooth holomorphic map $$f\,:\,\cX \,\to\, S\, ,$$ a \ka form
$\omega_\cX$ and a locally free sheaf $\cF$ on $\cX$. The determinant line bundle
of $\cF$ on $S$ is by definition the Knudsen-Mumford determinant line bundle
$\lambda\,=\,\det \ul{\ul R}f_* \cF$. Note that according to Grauert's theorem,
locally with respect to the base, the direct image $\ul{\ul R}f_* \cF$, as an
element of the derived category can be represented by a bounded complex
$\cL^\bullet$ of locally free sheaves on $S$ so that for the usual direct image
sheaves $R^qf_*\cF = \cH^q(\cL^\bullet)$ holds. The determinant line bundle
$\lambda(\cF)$ can be computed as determinant line bundle of the complex
$\cL^\bullet$. Bismut, Gillet, and Soulé identified it with a holomorphic line
bundle, for which they constructed a Quillen metric $h^Q$ and proved a generalized
Riemann-Roch theorem in the category of hermitian vector bundles, for $\dim
X\,=\,1$ see also the theorem of Zograf and Takhtadzhyan \cite{zt}.

\begin{theorem}[{Bismut-Gillet-Soulé \cite[Theorem 0.1]{bgs}}]
The Chern form of the determinant line bundle is equal to the fiber integral
\begin{equation}\label{eq:bgs}
c_1({\lambda,h^Q})= \left( \int_{\mathcal X/S} {\rm ch}(\cF,h) {\rm td}(\cX/S,\omega_{\cX})\right)_{(1,1)},
\end{equation}
where ${\rm ch}(\cF,h)$ and ${\rm td}(\cX/S,\omega_\cX)$ denote respectively the Chern character form for $(\cF,h)$ and the Todd character form for the relative
tangent bundle.
\end{theorem}

In this section, we apply the results to the case, where $X$ is a \ka manifold
whose \ka form $\omega_X$ is the Chern form $c_1(\cL, h_\cL)$ of a positive
hermitian line bundle $(\cL, h_\cL)$. The base of a holomorphic family $S$ will be provided with a flat \ka form (which is immaterial for our arguments) inducing a \ka
structure on $X \times S$. The relative \ka form is $\omega_{\cX/S}$.

\subsection{Virtual bundles and Chern character forms}

Given any hermitian holomorphic vector bundle $(\cE,h)$ of rank $r$ on $Z=X\times S$, we introduce the virtual hermitian bundle (i.e.\ element of the Grothendieck group)
$$
(End(\cE)- \cO_Z^{\oplus r^2})\otimes (\cL -\cO_{Z})^{\otimes(n-1)}.
$$
The (virtual) rank is zero, and the first Chern class vanishes. We identify $\omega_{X\times S/S}$ with the pullback $q^*\omega_X$, where $q:X\times S \to X$ is the projection, and simply write $\omega_X$. Furthermore the trivial bundles carry a flat metric.

\begin{lemma}\label{le:Ch}
We have the following identities for Chern character forms, where we only note the lowest degree terms.
$$
\mathrm{ch}\left((End(\cE)-\cO^{\oplus r^2}_{X\times S} )\otimes (\cL - \cO_{X\times S} )^{\otimes(n-1)}\right)= \left(2 r \,\mathrm{ch}_2(\cE)- c^2_1(\cE)\right) \we \omega^{n-1}_X + \ldots
$$
and
$$
\mathrm{ch}\left( (\Lambda^r(\cE) - \cO_{X\times S})^{\otimes 2}\otimes (\cL - \cO_{X\times S} )^{\otimes(n-1)}\right) = c^2_1(\cE)\we \omega^{n-1}_X + \ldots
$$
\end{lemma}

The base spaces $S$ are again subject to the above conditions from Section~\ref{se:def}.

\begin{definition}
Given a holomorphic vector bundle $\cE$ on $X \times S$, choose $x_0\in X$, and
denote by $m:X \times S \to \{x_0\}\times S$ the canonical projection
defined by $(x, s)\mapsto (x_0,\,s)$.
Define the pulled back bundle $\cE^0 :=m^* (\cE|\{x_0\} \times S)$ on $X \times S$, and
equip this bundle with the trivial hermitian metric.
\end{definition}

\begin{lemma}
For the lowest degree forms we have:
\begin{eqnarray}
&&\mathrm{ch}\left( (\Lambda^r(\cE) \otimes\Lambda^r(\cE^0)^{-1} - \cO_{X\times S} )\otimes (\cL - \cO_{X\times S} )^{\otimes n}\right) = c_1(\cE) \we \omega^{n}_X + \ldots\\
&&\mathrm{ch}\left( (\Lambda^r(\cE) \otimes\Lambda^r(\cE^0)^{-1} - \cO_{X\times S} )^{\otimes2}\otimes (\cL - \cO_{X\times S} )^{\otimes(n-1)}\right) = c^2_1(\cE) \we \omega^{n-1}_X + \ldots
\end{eqnarray}
\end{lemma}
We apply the above notions to a holomorphic family quiver bundles $\pmb{{\cR}}=\pmb{(\cE,\phi)}$ equipped with quiver vortex structures in an analogous way.

\begin{proposition}\label{pr:owp}
Let $\omega_X = c_1(\cL,h_\cL)$. Then the generalized \wp form satisfies the following equation.
\begin{eqnarray}\label{eq:owp}
\frac{1}{4\pi^2} \omega_{WP}&=& \bigg(\int_X \bigg(\sum_{\lambda\in Q'_0} \Big(-\frac{n_\lambda}{2 r_\lambda} \mathrm{ch}\left((End(\cE_\lambda)-\cO^{\oplus r^2_\lambda}_{X\times S} )\otimes (\cL - \cO_{X\times S} )^{\otimes(n-1)}\right) \\
&& - \frac{n_\lambda}{2 r_\lambda} \mathrm{ch}\left( (\Lambda^{r_\lambda}(\cE_\lambda)\otimes\Lambda^{r_\lambda}(\cE^0_\lambda)^{-1} - \cO_{X\times S})^{\otimes 2}\otimes (\cL - \cO_{X\times S} )^{\otimes(n-1)}\right)\nonumber \\
&& + \frac{\tau'_\lambda}{2\pi} \mathrm{ch}\left( (\Lambda^{r_\lambda}(\cE_\lambda)\otimes\Lambda^{r_\lambda}(\cE^0_\lambda)^{-1} - \cO_{X\times S})\otimes (\cL - \cO_{X\times S} )^{\otimes n}\right) \Big) \bigg)\mathrm{td}(X\times S/S) \bigg)_{(1,1)}\nonumber\\
&& + \frac{\ii}{2\pi}\pt\db\bigg(\frac{1}{2\pi}\int_X \sum_{a\in Q'_1} \tr (\phi_a \phi^*_a) \we \omega^n_X\bigg)\nonumber
\end{eqnarray}
\end{proposition}
\begin{proof}
In all Chern character forms the lowest possible degree is $(n+1,n+1)$ so that for the evaluation of the fiber integral only the constant $1$ of the Todd character form contributes. The rest follows from Lemma~\ref{le:Ch}.
\end{proof}

\subsection{Determinant line bundles}

We introduce the following determinant line bundles,
\begin{eqnarray}
\delta'_\lambda&=& \det \ul{\ul R}q_*\left((End(\cE_\lambda)-\cO^{\oplus r^2}_{X\times S} )\otimes (\cL - \cO_{X\times S} )^{\otimes(n-1)}\right)\\
\delta''_\lambda&=&\det \ul{\ul R}q_*\left((\Lambda^{r_\lambda}(\cE_\lambda)\otimes\Lambda^{r_\lambda}(\cE^0_\lambda)^{-1} - \cO_{X\times S})\otimes (\cL - \cO_{X\times S} )^{\otimes n}\right)\\
\delta'''_\lambda&=&\det \ul{\ul R}q_*\left((\Lambda^{r_\lambda}(\cE_\lambda)\otimes\Lambda^{r_\lambda}(\cE^0_\lambda)^{-1} - \cO_{X\times S})^{\otimes 2}\otimes (\cL - \cO_{X\times S} )^{\otimes (n-1)}\right)
\end{eqnarray}
These will be equipped with the corresponding Quillen metrics $h_\lambda'$,
$h_\lambda''$, and $h_\lambda'''$ respectively. Furthermore we set
\begin{equation}
 \chi = \frac{1}{2\pi}\int_X \sum_{a\in Q'_1} \tr (\phi_a \phi^*_a) \we \omega^n_X
\end{equation}
and the trivial bundle $\cO_S$ will be equipped with the hermitian metric $e^{-\chi}$.

The theorem of Bismut, Gillet, and Soulé together with Proposition~\ref{pr:owp}
will imply the main result of this section.

\begin{theorem}\label{th2}
Let a family of stable quiver bundles $\pmb{{\wt\cR}}\,=\,\pmb{(\wt\cE,\wt\phi)}$, parameterized by a base space $S$, be given such that the numbers $\tau'_\lambda/2\pi$
in \eqref{eq:HE} are
rational. Then there exists a hermitian line bundle $(\delta, h^Q_\delta)$, whose curvature form is equal to a integer multiple of the \wp form. The construction is functorial, i.e.\ compatible with base change, and $(\delta,\, h^Q_\delta)$ descends to the moduli space.
\end{theorem}

\begin{proof}
Since we take a deformation theoretic viewpoint we need the construction of the moduli space from base spaces $S$. By assumption all quiver bundles are stable, in particular simple. The automorphism group $\C^*$ acts on the fibers, namely by multiplication on all $\cE_{\lambda s}$ and on the holomorphic maps $\phi_a$. This action leaves the bundles $End(\cE_{\lambda s})$ and $\Lambda^{r_\lambda}(\cE_\lambda)\otimes\Lambda^{r_\lambda}(\cE^0_\lambda)^{-1}$ invariant. Hence the determinant line bundles $\delta'_\lambda$, $\delta''_\lambda$, and $\delta'''_\lambda$ can be patched together and descend to the moduli space. In view of the quiver vortex equation, we are left with an action of $S_1$ so that the hermitian metric $e^{-\chi}$ on the trivial line bundle descends to the moduli space as well.
\end{proof}

\section{Curvature of the $L^2$-metric}

For a holomorphic family of stable quiver bundles together with solutions of the quiver vortex equation the harmonic representatives $\mu_i$ of classes in the hypercohomology $\mathbb H^1(C^\bullet)$ were given in \eqref{eq:mui}. At this point we assume that the base space is smooth (even though some of the statements below can be given a more general meaning).

So far, we identified tangent vectors $\pt/\pt s^i$ of the base of a universal family of holomorphic quiver bundles with harmonic generalized \ks forms $\mu_i$. The $L^2$ inner product was shown to define a natural \ka structure. Our aim is to compute the curvature and arrive at a result only in terms of harmonic \ks forms.

\subsection{Identities for generalized harmonic \ks tensors}
We will need two kinds of Laplace operators on the space $C^{0,0}$ of differentiable sections of $\oplus _{\lambda\in Q'_0} \End(\cE_\lambda)$. We set
$$
\Box^0=(d^0)^*d^0
$$
in the sense of Lemma~\ref{le:ellcompl} and Lemma~\ref{le:adj}. The usual Laplace operator on this space will be denoted by
$$
\Box= \ol\pt^*\ol\pt.
$$
\begin{lemma}
Consider the horizontal components $(R^\lambda_{i\ol\jmath})_{\lambda\in Q'_0}$ of the curvature. Then
\begin{equation}\label{eq:muijq}
\mu_{i;\ol\jmath}= d^{\, 0}\!\left( (\sqrt{n_\lambda} R^\lambda_{i \ol\jmath})_{\lambda\in Q'_0}\right)
\end{equation}
\end{lemma}

\begin{proof}
\begin{gather*}
d^{\, 0}\left( (\sqrt{n_\lambda}R^\lambda_{i \ol\jmath})_{\lambda\in Q'_0} \right) =
\left((R^{ha}_{i\ol\jmath}\phi_a - \phi_a R^{ta}_{i\ol\jmath})_{a\in Q'_1}, (\sqrt{n_\lambda}\db R^\lambda_{i\ol\jmath})_{\lambda\in Q'_0}\right) =
\qquad \strut \\ \qquad \left((\phi_{a;\ol\jmath i} - \phi_{a;i\ol\jmath})_a, (\sqrt{n_\lambda}R^\lambda_{i\ol\jmath;\ol\beta} dz^\ol\beta )_\lambda \right)= \left( (-\phi_{a;i\ol\jmath})_a ,(\sqrt{n_\lambda}R^\lambda_{i\ol \beta; \ol\jmath})_\lambda )\right)= \mu_{i;\ol \jmath}
\end{gather*}
\end{proof}

\begin{corollary}
The following identity holds:
\begin{equation}
(d^{\, 0})^* (\mu_{i;\ol\jmath}) = \Box^0\left((\sqrt{n_\lambda}R^\lambda_{i\ol\jmath})_\lambda \right).
\end{equation}
\end{corollary}

For non-conjugate indices the definition implies the following symmetry.

\begin{lemma}
$$
\mu_{i;k}\,=\,\mu_{k;i}
$$
\end{lemma}

We will need the following identity.

\begin{lemma}\label{le:d0smuik}
$$
(d^{\, 0})^*(\mu_{i;k})\,=\,0\, .
$$
\end{lemma}

\begin{proof}
\begin{gather*}
(d^{\, 0})^*(\mu_{i;k})= \left( \frac{1}{\sqrt{n_\lambda}}\left( -\sum_{a\in \inv h(\lambda)}\phi_{a;i}\phi^*_a+ \sum_{a\in \inv t(\lambda)}\phi^*_a \phi_{a;i}\right) + \sqrt{n_\lambda}\db^*(R^\lambda_{i\ol\beta}dz^{\ol\beta})\right)_{\lambda\in Q'_0}=\\
\left(\frac{1}{\sqrt{n_\lambda}}\left(-\sum_{a\in \inv h(\lambda)}(\phi_{a}\phi^*_a)_{;i}+ \sum_{a\in \inv t(\lambda)}(\phi^*_a \phi_{a})_{;i} - n_\lambda (g^{\ol\beta\alpha}R^\lambda_{\alpha\ol\beta})_{;i}\right)\right)_{\lambda\in Q'_0}=0
\end{gather*}
by the above definitions and \eqref{eq:HE}.
\end{proof}

\begin{definition}
A symmetric exterior product $[\mspace\we\mspace] :\wt C^1 \times \wt C^1 \longrightarrow \wt C^2$ is defined by
$$
\Big[(\chi_a, \xi_\lambda)_{a,\lambda}\we(\psi_a,\zeta_\lambda )_{a,\lambda}\Big]
$$
$$
= \left(\frac{1}{\sqrt{n_{ha}}}( \xi_{ha}\circ \psi_a + \zeta_{ha}\circ \chi_a) -\frac{1}{\sqrt{n_{ta}}} (\psi_a\circ \xi_{ta}+\chi_a \circ\zeta_{ta})\, , \, \frac{1}{\sqrt{n_\lambda}} [\xi_\lambda \we \zeta_\lambda ]\right)_{a,\lambda}\, ,
$$
where the second component of the product is given by the Lie product on the bundles $End(\cE_\lambda)$ together with the wedge product of alternating forms.
\end{definition}

\begin{lemma}\label{le:wepr}
$$
d^1(\mu_{i;k})+ [\mu_i \we \mu_k]=0
$$
\end{lemma}

\begin{proof}
\begin{eqnarray*}
d^1(\mu_{i;k})&=& d^1(-\phi_{a:ik} , \sqrt{n_\lambda} R^\lambda_{i\ol\beta;k}dz^\ol\beta)_{a,\lambda}\\
&=& \big( (-\phi_{a;ik\ol\beta}dz^\ol\beta - R^{ha}_{i\ol\beta;k}\phi_a + \phi_a R^{ta}_{i\ol\beta;k})dz^\ol\beta , \sqrt{n_\lambda} R^\lambda_{i\ol\beta;k\ol\delta} dz^\ol\delta\we dz^\ol\beta\big)_{a,\lambda}\\
&=& \big( (-\phi_{a;i\ol\beta k}+ R^{ha}_{k\ol\beta}\phi_{a;i} - \phi_{a;i}R^{ta}_{k\ol\beta} -R^{ha}_{i\ol\beta;k}\phi_a + \phi_a R^{ta}_{i\ol\beta;k})dz^\ol\beta,\\
&&\sqrt{n_\lambda}(R^\lambda_{i\ol\beta;\ol\delta k} -[R^\lambda_{k\ol\delta},R^\lambda_{i\ol\beta}])dz^\ol\delta\we dz^\ol\beta \big)_{a,\lambda}\\
 &=& \Big(\big( - (\phi_{a;\ol\beta i} - R^{ha}_{i\ol\beta}\phi_a + \phi_a R^{ta}_{i\ol\beta})_{;k} + R^{ha}_{k\ol\beta}\phi_{a;i}
\phi_{a;i}R^{ta}_{k\ol\beta} -R^{ha}_{i\ol\beta;k}\phi_a
+ \phi_a R^{ta}_{i\ol\beta;k}\big)dz^\ol\beta,\\
&&
\sqrt{n_\lambda} ([R^\lambda_{k\ol\delta}, R^\lambda_{i\ol\beta}]) dz^\ol\delta\we dz^\ol\beta
\Big)_{a,\lambda}\\
&=& \Big( (R^{ha}_{i\ol\beta}\phi_{a;k}- \phi_{a;k} R^{ta}_{i\ol\beta} + R^{ha}_{k\ol\beta}\phi_{a;i} - \phi_{a;i}R^{ta}_{k\ol\beta})dz^\ol\beta, - \sqrt{n_\lambda} ([R^\lambda_{k\ol\delta}, R^\lambda_{i\ol\beta}]) dz^\ol\delta\we dz^\ol\beta\Big)_{a,\lambda}\\
&=& -[\mu_i\we \mu_k].
\end{eqnarray*}
\end{proof}

\begin{lemma}\label{le:d*dxi}
For any section $(\xi_\lambda)_\lambda$ in $C^{0,0}$
$$
 \Box^0(\xi_\lambda)= \Bigg(
\sum_{a\in \inv t(\lambda)}\Big( \frac{1}{n_{\lambda}} \phi^*_a \phi_a \xi_\lambda
-\frac{1}{\sqrt{n_\lambda n_{h(a)}}} \phi^*_a \xi_{h(a)}\phi_a\Big)
$$
$$
+\sum_{a\in \inv h (\lambda)} \Big(\frac{1}{n_\lambda} \xi_\lambda \phi_a \phi^*_a - \frac{1}{\sqrt{n_\lambda n_{t(a)}}} \phi_a \xi_{t(a)}\phi^*_a\Big) + \Box(\xi_\lambda)
\Bigg)_{\!\!\lambda}
$$
holds.
\end{lemma}

The lemma follows immediately from the definition of $d^{\,0}$ and $(d^{\,0})^*$.

\begin{lemma}\label{le:BoxR}
$$
\Box^0(\sqrt{n_\lambda}R^\lambda_{i\ol\jmath})= \bigg(-\sqrt{n_\lambda}g^{\ol\beta\alpha} \big[ R^\lambda_{\alpha\ol\jmath} ,R^\lambda_{i\ol\beta} \big] +\frac{1}{\sqrt{n_\lambda}}\big(\sum_{a\in \inv h(\lambda)} \phi_{a;i}\phi^*_{a;\ol\jmath} - \sum_{a\in \inv t(\lambda)} \phi^*_{a;\ol\jmath}\phi_{a;i}\big)
\bigg)_{\!\!\lambda}
$$
\end{lemma}

\begin{proof}
We set $\xi_\lambda= \sqrt{n_\lambda}R^\lambda_{i\ol\jmath}$ in Lemma~\ref{le:d*dxi}.
\begin{gather*}
\Box(\sqrt{n_\lambda}R^\lambda_{i\ol\jmath})=
\Bigg(\frac{1}{\sqrt{n_\lambda}} \sum_{a\in \inv t(\lambda)} \Big( \phi^*_a\phi_a R^\lambda_{i\ol\jmath} - \phi^*_a R^{h(a)}_{i\ol\jmath}\phi_a \Big) + \\ \hspace{4cm} \frac{1}{\sqrt{n_\lambda}} \sum_{a\in \inv h (\lambda)} \Big( R^\lambda_{i\ol\jmath} \phi_a\phi^*_a- \phi_a R^{t(a)}_{i\ol\jmath} \phi^*_a \Big)
-\sqrt{n_\lambda}g^{\ol\beta\alpha} R^\lambda_{i\ol\jmath;\ol\beta\alpha}g^{\ol \beta\alpha}\Bigg)_{\!\!\lambda}.
\end{gather*}
From \eqref{eq:HE} we get
\begin{gather*}
\big(-n_\lambda g^{\ol\beta\alpha} R^\lambda_{i \ol\jmath;\ol\beta\alpha}\big)_\lambda =
\bigg(\sum_{a\in \inv h(\lambda)}\big( \phi_{a;i} \phi^*_{a;\ol\jmath} - R^\lambda_{i\ol\jmath}\phi_a\phi^*_{a} + \phi_a R^{t(a)}_{i\ol\jmath} \phi^*_{a}\big) \\ \hspace{4cm}
+ \sum_{a\in \inv t(\lambda)}\big(- \phi^*_{a;\ol\jmath}\phi_{a;i} + \phi^*_a R^{h(a)}_{i\ol\jmath}\phi_a - \phi^*_a \phi_a R^\lambda_{i\ol\jmath}
\big)
\bigg)_\lambda
\end{gather*}
which yields the lemma.
\end{proof}

\begin{definition}
A hermitian form $(\mspace\vee\mspace ): \wt C^1 \times \wt C^1 \to \wt C^0$ is defined by
$$
\Big((\psi_a, \xi_{\lambda,\ol\beta}dz^{\ol\beta}) \vee(\chi_a,\eta_{\lambda,\alpha}dz^\alpha)^*\Big)_\lambda= \Big(\sum_{a\in \inv h(\lambda)}\psi_a \chi^*_a - \sum_{a\in \inv t(\lambda)}\chi^*_a \psi_a - g^{\ol\beta\alpha}[ \xi_{\lambda,\ol\beta} ,\eta^*_{\lambda, \alpha} ]\Big)_\lambda
$$
\end{definition}

Now the statement of Lemma~\ref{le:BoxR} has the following form:

\begin{lemma}
$$
(d^{\,0})^*d^{\,0} \left((n_\lambda R^\lambda_{i\ol\jmath})_\lambda\right)= (\mu_i \vee \mu_\ol\jmath).
$$
\end{lemma}

\subsection{\wp metric -- connection form}

The metric tensor was given in \eqref{eq:wp}. We will use the notation
$\mspace|_k$ for derivatives $\pt_k = \pt/\pt s^k$ with respect to coordinates
on the base.

\begin{proposition}\label{pr:Gijk}
\begin{equation}
\pt_k G^{WP}_{i\ol\jmath}= \int_X \tr(\mu_{i;k}\mu^*_\ol\jmath)g\, dV\, .
\end{equation}
\end{proposition}

\begin{proof}
The integral
$$
\int_X\tr(\mu_i \mu^*_{\ol\jmath;k})g\/ dV
$$
is the $L^2$ product of $\mu_i$ and $d^{\,0}(\sqrt{n_\lambda}R^\lambda_{j\ol k})$ by \eqref{eq:muijq}. The harmonicity of $\mu_i$, i.e.\ $(d^{\,0})^*(\mu_i)=0$, according to Proposition~\ref{pr:mui} implies that $\int_X\tr(\mu_i \mu^*_{\ol\jmath;k})g\/ dV$ vanishes.
\end{proof}
At a given point $s_0\in S$ we introduce holomorphic normal coordinates of the second kind, which means that the \ka form coincides with the standard unitary one
in terms of the coordinate chart, up to order two at $s_0$, or equivalently, the partial derivatives of the local expression of the Hermitian metric vanish at $s_0$. Proposition~\ref{pr:Gijk} implies that this condition is equivalent to
the condition that the harmonic projections of all $\mu_{i;k}$ vanish (at $s_0$):
$$
H(\mu_{i;k}(s_0))=0 \text{ for all } i \text{ and } k.
$$
We denote by $G$ the (abstract) Green's operator on various levels for $\wt C^\bullet$.

\begin{lemma}\label{le:Grmuik}
At $s=s_0$ we have
\begin{equation}
\mu_{i;k}=- (d^1)^* G([\mu_i\we\mu_k]).
\end{equation}
\end{lemma}

\begin{proof}
$$
\mu_{i;k}= G\left((d^1)^*d^1 + d^0 (d^{\,0})^*\right)(\mu_{i;k})
$$
Now Lemma~\ref{le:d0smuik} and Lemma~\ref{le:wepr} together imply the lemma.
\end{proof}

\subsection{Computation of the curvature}
In terms of normal coordinates the curvature tensor is given by
$$
-R^{WP}_{i\ol\jmath k \ol\ell}= G^{WP}_{i\ol\jmath| k \ol\ell}= \int_X\tr(\mu_{i;k\ol\ell}\cdot \mu^*_\ol\jmath ) \, g\, dV + \int_X \tr(\mu_{i;k}\cdot \mu^*_{\ol\jmath;\ol\ell}) \, g\, dV .
$$
The second integral can be computed using Lemma~\ref{le:Grmuik}. It equals
$$
\int_X \tr\left( G([\mu_i\wedge \mu_k]) \cdot
[\mu^*_{\ol\jmath}\wedge\mu^*_{\ol\ell}] \right)\, g\, dV.
$$
We compute the first integral. We will need
$$
\mu_{i;k\ol\ell}= \big(-\phi_{a;ik\ol\ell}, \sqrt{n_\lambda}R^\lambda_{i\ol\beta;k\ol\ell}\big)_{a,\lambda},
$$
and the following simple consequences of the Ricci identities.
\begin{eqnarray}
-\phi_{a;ik\ol\ell} &=& R^{h(a)}_{i\ol \ell;k} \phi_a - \phi_a R^{t(a)}_{i\ol \ell;k} + R^{h(a)}_{i\ol\ell} \phi_{a;k}- \phi_{a;k}R^{t(a)}_{i\ol\ell} + R^{h(a)}_{k\ol\ell} \phi_{a;i}- \phi_{a;i}R^{t(a)}_{k\ol\ell}\\
R^\lambda_{i\ol\beta;k\ol\ell} &=& R^\lambda_{i\ol\ell;k\ol\beta} +[R^\lambda_{k\ol\beta},R^\lambda_{i\ol\ell}] +[R^\lambda_{i\ol\beta},R^\lambda_{k\ol\ell}]
\end{eqnarray}
Next the quiver vortex equation (and partial integration) imply
\begin{equation}\label{eq:Rilkbeta}
n_\lambda\int_X \tr\big( R^\lambda_{i\ol\ell;k\ol\beta} R^\lambda_{\alpha\ol\jmath} g^{\ol\beta\alpha}\big)\, g\, dV = \int_X \tr\Big(R^\lambda_{i \ol\ell;k}\Big( \sum_{a \in \inv h(\lambda)}\phi_a\phi^*_{a;\ol\jmath} -\sum_{a \in \inv t(\lambda)} \phi^*_{a;\ol\jmath} \phi_a \Big)\Big)\,g\, dV
\end{equation}

\begin{proposition}
Given a universal family of stable holomorphic quiver bundles, the curvature tensor of the generalized \wp metric equals
\begin{eqnarray}
 R^{WP}_{i\ol\jmath k \ol\ell}& = & -\int_X \tr\Big( G([\mu_i\we\mu_k]) [\mu^*_\ol\jmath\we\mu^*_\ol\ell]\Big)\,g\,dV \nonumber \\
 && + \int_X \sum_{\lambda\in Q'_0} n_\lambda \tr\big((d^{\,0})^*d^{\,0}(R^\lambda_{i\ol\jmath}) R^\lambda_{k\ol\ell}\big) \,g\,dV\\ \nonumber
 && + \int_X \sum_{\lambda\in Q'_0} n_\lambda \tr\big((d^{\,0})^*d^{\,0}(R^\lambda_{k\ol\jmath}) R^\lambda_{i\ol\ell}\big) \,g\,dV
\end{eqnarray}
\end{proposition}

\begin{proof}
We use the following (local) short-hand notation:
$$
[\xi_\lambda,\chi_a]:= \xi^{h(a)} \chi_a - \chi_a \xi^{t(a)}.
$$
Now, again by the holomorphicity of $\phi_a$, and the Ricci identities, we get
$$
\tr(\phi_{a;ik\ol\ell}\phi^*_{a;\ol\jmath})=-\tr\Big( [R^\lambda_{i\ol\ell;k}, \phi_a] + R^\lambda_{i\ol\ell} [\phi_{a;k}, \phi^*_{a;\ol\jmath}] + R^\lambda_{k\ol\ell} [\phi_{a;i}, \phi^*_{a;\ol\jmath}]
\Big)
$$
The contribution of the second components is
\begin{eqnarray*}
n_\lambda \tr\Big(g^{\ol\beta\alpha} R^\lambda_{i\ol\beta; k\ol\ell}R^\lambda_{\alpha\ol\jmath} \Big)&=& n_\lambda \tr\Big(g^{\ol\beta\alpha}\big( R^\lambda_{i\ol\ell;k\ol\beta} + [R^\lambda_{k\ol\beta}, R^\lambda_{i\ol\ell} ]+[R^\lambda_{i\ol\beta}, R^\lambda_{k\ol\ell} ] \big)R^\lambda_{\alpha\ol\jmath} \Big)\\
&=& n_\lambda \tr\Big(g^{\ol\beta\alpha}\big( R^\lambda_{i\ol\ell;k\ol\beta}+ [R^\lambda_{\alpha\ol\jmath}, R^\lambda_{k\ol\beta}]R^\lambda_{i\ol\ell} + [R^\lambda_{\alpha\ol\jmath}, R^\lambda_{i\ol\beta}]R^\lambda_{k\ol\ell}
\Big)
\end{eqnarray*}
We modify the first term by adding a term of divergence type and use \eqref{eq:HE}. We obtain
\begin{gather*}
n_\lambda \tr\Big(g^{\ol\beta\alpha}\big(- R^\lambda_{i\ol\ell;k}R^\lambda_{\alpha\ol\jmath;\ol\beta} +[R^\lambda_{\alpha\ol\jmath}, R^\lambda_{k\ol\beta}]R^\lambda_{i\ol\ell} + [R^\lambda_{\alpha\ol\jmath}, R^\lambda_{i\ol\beta}]R^\lambda_{k\ol\ell}
\Big)\\
= \tr\big( - R^\lambda_{i\ol\ell;k} [\phi_a,\phi^*_{a;\ol\jmath} ]\big) + n_\lambda \tr\Big([R^\lambda_{\alpha\ol\jmath}, R^\lambda_{k\ol\beta}]R^\lambda_{i\ol\ell} + [R^\lambda_{\alpha\ol\jmath}, R^\lambda_{i\ol\beta}]R^\lambda_{k\ol\ell} \Big)
\end{gather*}
These two equations imply
\begin{gather*}
\int_X\tr\Big(\sum_{a\in Q'_1} \phi_{a;ik\ol\ell}\phi^*_{a;\ol\jmath} + \sum_{\lambda\in '_0} g^{\ol\beta\alpha} R^\lambda_{i\ol\beta; k\ol\ell}R^\lambda_{\alpha\ol\jmath}
\Big)\,g\,dV = \hspace{6cm}\\
\int_X\tr\Big(-\sum_{a\in Q'_1}\big(R^\lambda_{i\ol\ell} [\phi_{a;k},\phi^*_{a;\ol\jmath}]+ R^\lambda_{k\ol\ell} [\phi_{a;i},\phi^*_{a;\ol\jmath}]\big)+ \sum_{\lambda\in '_0} g^{\ol\beta\alpha} R^\lambda_{i\ol\beta; k\ol\ell}R^\lambda_{\alpha\ol\jmath}
\Big)\,g\,dV + \\
n_\lambda\int_X \sum_{a\in Q'_1}\tr\Big( R^\lambda_{i\ol\ell}[\phi_{a;k}, \phi^*_{a;\ol \jmath} ] +R^\lambda_{k\ol\ell}[\phi_{a;i}, \phi^*_{a;\ol \jmath} ] \Big) \,g\,dV.
\end{gather*}
Finally Lemma~\ref{le:BoxR} implies the proposition.
\end{proof}

We denote the Green's operator on the spaces $\wt C^k$ by the letter $G$. As above we assume that the base of a universal local family is smooth. Note that in the following formula the quantity $\mu_i\we\mu_j$ need not vanish in dimension one -- only the contribution from the second component does, and it is necessary to assume that the first component vanishes, in order to get positive curvature for $\dim X= 1$.

\begin{theorem}\label{tf}
The curvature of the generalized \wp metric equals
\begin{eqnarray*}
R^{WP}_{i\ol\jmath k\ol\ell}&=&
-\int_X \left( G([\mu_i\wedge \mu_k]) \cdot [\mu^*_{\ol\jmath}\wedge\mu^*_{\ol\ell}] \right)\, g\, dV\\
&&+ \int_X \left( G([\mu_i\vee \mu_\ol\jmath]) \cdot [\mu^*_{k}\vee\mu^*_{\ol\ell}] \right)\, g\, dV\\
&&+ \int_X \left( G([\mu_k\vee \mu_\ol\jmath]) \cdot [\mu^*_{i}\vee\mu^*_{\ol\ell}] \right)\, g\, dV
\end{eqnarray*}
\end{theorem}

\section*{Acknowledgements}

We are grateful to the referee for detailed comments to improve the exposition.
The first-named author acknowledges the support of a J. C. Bose Fellowship.

\end{document}